\documentclass{article}

\usepackage{natbib} 

\usepackage{tikz}
\usepackage{pifont}

\usepackage[hidelinks]{hyperref} 

\usepackage[utf8]{inputenc} 
\usepackage[T1]{fontenc}    
\usepackage{hyperref}       
\usepackage{url}            
\usepackage{booktabs}       
\usepackage{amsfonts}       
\usepackage{nicefrac}       
\usepackage{microtype}      
\usepackage{xcolor}         
\usepackage{float} 
\usepackage{graphicx} 
\usepackage{subfigure} 

\usepackage{bm} 

\usepackage{makecell} 

\usepackage[margin=1.03in]{geometry} 

\usepackage{algorithm} 
\usepackage{algorithmic}

\usepackage{wrapfig}

\usepackage{enumerate} 

\usepackage{amsmath}
\usepackage{amssymb}
\usepackage{mathtools}
\usepackage{amsthm}
\usepackage{mathrsfs}

\usepackage{threeparttable}

\newtheorem{theorem}{Theorem}

\newtheorem{lemma}{Lemma}
\newtheorem{remark}{Remark}

\newtheorem{proposition}{Proposition}
\newtheorem{assumption}{Assumption}

\newtheorem*{theorem*}{Theorem}
\newtheorem*{example*}{Example} 
\newtheorem*{definition*}{Definition}
\newtheorem*{lemma*}{Lemma}
\newtheorem*{remark*}{Remark}
\newtheorem*{corollary*}{Corollary}
\newtheorem*{proposition*}{Proposition}
\newtheorem*{assumption*}{Assumption}
\newtheorem*{claim*}{Claim}

\renewcommand{\(}{ \left( }
\renewcommand{\)}{ \right) }

\newcommand{\<}{ \left< }
\renewcommand{\>}{ \right> }

\renewcommand{\Pr}{ \mathbb{P} }

\newcommand{\R}{\mathbb{R}}

\newcommand{\E}{\mathbb{E}}

\newcommand{\ene}{\mathcal{E}}

\newcommand{\lnorm}{\left\|}
\newcommand{\rnorm}{\right\|}

\newcommand{\lb}{\left(}
\newcommand{\rb}{\right)}
\newcommand{\lbm}{\left[}
\newcommand{\rbm}{\right]}

\newcommand{\N}{\mathbb{N}}

\title{Breaking a Logarithmic Barrier in the \emph{Stopping Time} Convergence Rate of Stochastic First-order Methods} 

\author{
Yasong Feng\thanks{Shanghai Center for Mathematical Sciences, Fudan University; email: \texttt{ysfeng20@fudan.edu.cn}.}\and
Yifan Jiang\thanks{Mathematical Institute, University of Oxford; email: \texttt{yifan.jiang@maths.ox.ac.uk}.}\and
Tianyu Wang\thanks{Shanghai Center for Mathematical Sciences, Fudan University; email: \texttt{wangtianyu@fudan.edu.cn}.}\and
Zhiliang Ying\thanks{Department of Statistics, Columbia University; email: \texttt{zying@stat.columbia.edu}.}}
\date{}
\begin{document}

\maketitle

\begin{abstract}
This work provides a novel convergence analysis for stochastic optimization in terms of stopping times, addressing the practical reality that algorithms are often terminated adaptively based on observed progress. Unlike prior approaches, our analysis: 1. Directly characterizes convergence in terms of stopping times adapted to the underlying stochastic process. 2. Breaks a logarithmic barrier in existing results. Key to our results is the development of a lemma to control the large deviation property of almost super-martingales. This lemma might be of broader interest. 
\end{abstract}




\section{Introduction}

In practice, people stop stochastic algorithms when the progress
meets certain criteria -- the time to terminate a stochastic algorithm is a stopping time. However, existing analyses of stochastic algorithms fail to directly capture this fundamental aspects of real-world algorithm deployment -- they either focus on asymptotic convergence behavior \cite[e.g.,][]{robbins1951stochastic,kiefer1952stochastic,doi:10.1137/0319007,bertsekas2000gradient,sebbouh2021almost}, or convergence rate in terms of a non-random iterate count \cite[e.g.,][]{nemirovski2009robust,lan2012optimal,harvey2019tight,lan2020first,liu2020improved,liu2023high,liu2023revisiting}. 
%
This critical gap between theory and practice motivates our work. 

\subsection{The Stopping Time Perspective}
In virtually all applications of stochastic optimization -- from training machine learning models to solving large-scale operations research problems -- algorithms are terminated based on dynamic stopping rules rather than a non-random iteration count. These stopping times typically take forms such as:
\begin{align*}
\tau &= \min \{ k : \|x_k - x_{k-1}\| \leq \varepsilon \} \quad \text{or} \quad \\
\tau &= \min \{ k : |f(x_k) - f(x_{k-1})| \leq \varepsilon \}, 
\end{align*}
where $\varepsilon >0$ is some acceptable tolerance level (e.g., $\varepsilon = 0.0001$), $f$ is an objective function, and $\{ x_k \}_k$ is governed by the stochastic algorithm. Such stopping rules are widely employed in practice -- See the works on early stopping \cite[e.g.,][]{prechelt2002early,dodge2020fine}, which are used as practical guidelines in machine learning and deep learning. 
This adaptive termination criterion reflects the reality that different problem instances and random seeds may require dramatically different numbers of iterations to reach comparable solution quality. 

Surprisingly, despite the wide adoption of such stopping rules in applications, the theoretical literature has almost exclusively focused on convergence guarantees for non-random iteration counts. This creates a significant gap between the theoretical understanding of algorithm performance and how algorithms are actually used. 

\subsection{State-of-the-Art Results and Limitations} 
The best known high-prob-ability convergence rate for stochastic first-order methods on smooth convex objectives,\textit{ when the learning rate does not depend on the {total} number of iterations}, takes the following form \citep{lan2012optimal,lan2020first,liu2023high,liu2023revisiting} (second item of Theorem 3.6 in \citep{liu2023high}, first item of Theorem 3.3 in \citep{liu2023revisiting}): 
\begin{align}
    \Pr\left( f(x_k) - f^* \lesssim \tfrac{\log k}{\sqrt{k}} \cdot \mathrm{\text{poly-log}} (\tfrac{1}{\beta})  \right) \geq 1-\beta, \quad \forall \beta \in (0,1), \; k \in \mathbb{N}_+, \label{eq:prev} 
\end{align}
where $ f^* > -\infty $ is the minimum of $ f $, $ \{ x_k \}_k $ is the sequence of iterates governed by the algorithm, and $ \lesssim $ omits constants that does not depend on $k$ or $\beta$. 

By applying a union bound to (\ref{eq:prev}) over $ k $, we obtain (See Section \ref{sec:prior-detail} for details):
\begin{align} 
    \Pr\left( \forall k \in \mathbb{N}_+, \; f(x_k) - f^* \lesssim  \frac{ \log k }{ \sqrt{k} } \cdot \mathrm{\text{poly-log}} (\tfrac{1}{\beta}) + \frac{\log^2 k}{\sqrt{k}} \right) \ge 1-\beta, \;\; \forall \beta \in (0,1), \label{eq:prev-union} 
\end{align}
or equivalently (See Lemma \ref{lemma:iff}), for any $ \beta \in (0,1), \; \text{and any stopping time } \tau$,  
\begin{align} 
    \Pr\left( f(x_\tau) - f^* \lesssim   \frac{ \log \tau }{ \sqrt{\tau } } \cdot \mathrm{\text{poly-log}} (\tfrac{1}{\beta}) + \frac{\log^2 \tau }{\sqrt{\tau }} \right) \ge 1-\beta. \label{eq:prev-stopping} 
\end{align} 




This analysis exposes a fundamental limitation in current state-of-the-art approaches:

\begin{itemize}
\item \textbf{The Adaptivity Gap}: When moving from guarantees like \eqref{eq:prev} to stopping-time guarantees \eqref{eq:prev-stopping}, we incur an extra logarithmic factor.
\end{itemize}

This limitation prompts several critical theoretical questions in stochastic optimization and stochastic approximation:

\begin{quote}
\textbf{(Q)} \emph{Can we bridge this adaptivity gap? Specifically:
\begin{itemize}
    \item[] Is the $ \log^2 \tau$ dependence in \eqref{eq:prev-stopping} an artifact of current proof techniques, or does it reflect an inherent computational barrier? Can we achieve or surpass the ``ideal'' $ \tau^{-1/2} \log \tau$ rate for stopping time convergence rate, matching what's best known for non-random iteration counts? 
\end{itemize}}
\end{quote}



\subsection{Our results}

The convergence rates established in (\ref{eq:prev}) represent a pinnacle achievement in stochastic optimization theory, making \textbf{(Q)} a significant challenge. 
Consequently, it may seem that there is little opportunity for improvement in (\ref{eq:prev-stopping}). However, in this work, we contest this assumption and achieve an improvement by surpassing a logarithmic barrier. Specifically, we establish Theorem \ref{thm:main} and Proposition \ref{prop}; A comparison with state-of-the-art results are summarized in Table \ref{tab:lip}.

\begin{theorem} 
    \label{thm:main}
    Let the objective function $f : \R^n \to \R$ be convex, $L$-smooth, and admits a minimum $f^* > -\infty$ at some $x^* \in \R^n$.  
    Instate Assumptions \ref{assump:sto-gradient} and \ref{assump:hp} (standard conditions for stochastic first-order methods specified below). 
    A variant of the classic Stochastic Gradient Descent with Momentum (SGDM) algorithm \citep{lan2012optimal} (stated in Eq. \ref{eq:SGDM}) satisfies the following: the sequence $\{ x_k \}_k$ governed by this algorithm satisfies: $\forall \beta \in (0,0.5) $, 
    \begin{align*} 
        \Pr\left(f(x_k)-f^*\leq\frac{\left(C_1+C_2\log\frac{1}{\beta}\right)\log(k+2)}{\sqrt{k+1}},\;\text{for all $k\in \mathbb{N}_+$}\right)\geq1-2\beta, 
    \end{align*} 
    and \text{for any $\beta \in (0,0.5)$ and any stopping time $\tau$}, 
    \begin{align*}
        \Pr\left(f(x_{\tau})-f^*\leq\frac{\left(C_1+C_2\log\frac{1}{\beta}\right)\log(\tau+2)}{\sqrt{\tau +1}}\right)\geq1-2\beta, 
    \end{align*}
    where $C_1$ and $C_2$ are absolute constants depending only on the problem parameters (e.g., smoothness $L$) and the starting point $x_0$. 
\end{theorem} 

In addition, we have the following variant of Theorem \ref{thm:main}. 

\begin{proposition} 
    \label{prop}
    Instate the conditions in Theorem \ref{thm:main}. 
    For any $\varepsilon \in (0,0.5)$, a variant of the classic Stochastic Gradient Descent with Momentum (SGDM) algorithm \citep{lan2012optimal} (stated in Eq. \ref{eq:SGDM}) with $\varepsilon$-dependent step-size satisfies the following: the sequence $\{ x_k \}_k$ governed by this algorithm satisfies, for any $ \beta \in (0,0.5) $, 
    \begin{align*} 
        \Pr\left(f(x_k)-f^*\leq C_0 \cdot \frac{  h^\sigma (\varepsilon) \( 1 + \log\frac{1}{\beta} \)   \log^{\frac{1+\varepsilon}{2} }(k+2)}{\sqrt{k+1}},\;\text{for all $k\in \mathbb{N}_+$}\right)\geq1-2\beta, 
    \end{align*} 
    and \text{for any $\beta \in (0,0.5)$ and any stopping time $\tau$}, 
    \begin{align*}
        \Pr\left(f(x_{\tau})-f^*\leq C_0 \cdot \frac{ h^\sigma (\varepsilon) \( 1 + \log\frac{1}{\beta} \)  \log^{\frac{1 + \varepsilon}{2} }(\tau+2)}{\sqrt{\tau +1}}\right)\geq1-2\beta, 
    \end{align*}
    where $ h^\sigma (\varepsilon) := \exp \( \sigma^2 \zeta ( 1 + \varepsilon ) \) (\zeta ( 1 + \varepsilon ))^2 $, $\zeta$ is the Riemann zeta function, and $C_0$ is an absolute constants depending only on the problem parameters (e.g., smoothness $L$) and the starting point $x_0$.  
\end{proposition} 



\begin{table*}[t]
    \centering
    \begin{threeparttable}
        \caption{Comparison of high probability convergence rates for stochastic optimization algorithms. The second column contains the error bounds with probability at least $1-\beta$. In the third column, `Non-random iterate $k$' means that the error bound is for $f(x_k)-f^*$, where $k$ is a non-random iterate. `Stopping time $\tau$' means that the error bound holds for $f(x_\tau)-f^*$, where $\tau$ is an $\{x_k\}$-stopping time. 
        } 
        \label{tab:lip}
        \begin{tabular}{|c|c|c|c|}        
            \hline 
             & Convergence rate & Type & Domain \\
            \hline 
            \cite{lan2012optimal,lan2020first, liu2023high}\tnote{1}
            & $\frac{1}{\sqrt{k}}\cdot{\sqrt{\log\frac{1}{\beta}} \cdot \log k}$ 
            & {\makecell{Non-random\\iterate $k$.}}
            & $\mathbb{R}^n$ 
            \\ \hline
            
            \cite{lan2012optimal,lan2020first, liu2023high}\tnote{2} 
            & $\frac{1}{\sqrt{\tau}}\left({{\log^2\tau} + \log\frac{1}{\beta} \cdot \log\tau}\right)$ 
            & Stopping time $\tau$ 
            & $\mathbb{R}^n$ 
            \\ \hline
            
            \textbf{\makecell{This work\\(Theorem \ref{thm:main})}} 
            & $\frac{1}{\sqrt{\tau}}\cdot{\log\frac{1}{\beta}\log\tau}$
            & Stopping time $\tau$ 
            & $\mathbb{R}^n$ 
            \\ \hline
            \textbf{\makecell{This work\\(Proposition \ref{prop})}}\tnote{3}  
            & \makecell{$\frac{ h (\varepsilon) (1 + \log\frac{1}{\beta})\log^{\frac{1+\varepsilon}{2} } \tau }{\sqrt{\tau}}, $\\ for any $\varepsilon \in (0,0.5)$.}  
            & Stopping time $\tau$ 
            & $\mathbb{R}^n$ 
            \\ \hline
        \end{tabular}
        \begin{tablenotes}
            \item[1] The first row shows results only for learning rates independent of the \textbf{total} iteration count; learning rates that depend on the \textbf{total} iteration count are ill-posed in our setting. 
            \item[2] The rate in this row is derived from the high probability bound (\ref{eq:up-liu}) from \cite{liu2023high} and a union bound; See Section \ref{sec:prior-detail} for details.
            \item[3] Here $ h (\varepsilon) := \exp \( \zeta ( 1 + \varepsilon ) \) (\zeta ( 1 + \varepsilon ))^2 $, with $\zeta$ being the Riemann $\zeta$ function, only depends on $\varepsilon$. 
        \end{tablenotes}
    \end{threeparttable}
\end{table*}

\begin{remark}
    The convergence rate's dependence on the stopping time $\tau$ is fundamentally more significant than its dependence on the confidence level $\beta$ and the parameter $\varepsilon$. 
    In practice, $\beta$ is a constant (e.g., $\beta = 0.001$) and $\varepsilon$ is a parameter, while $\tau$ is a random stopping time supported on $\mathbb{N}_+$ and may be arbitrarily large. 
\end{remark} 

More importantly, we prove the following lemma that controls the large deviation properties of almost super-martingales, which might be of broader interest.

    
\begin{lemma}  
    \label{lem:gronwall}
    Let $\sigma > 0$ and $B\in(0,1]$ be two constants, and $\{ a_k \}_{k \in \mathbb{N}_+}$ be a nonnegative constant sequence. Consider stochastic processes $ \{ \mathcal{E} (k) \}_{k \in \mathbb{N}} $, $ \{ \theta_k \}_{k \in \mathbb{N}} $, $ \{ \varphi_k \}_{k \in \mathbb{N}} $ and $ \{ \kappa_k \}_{ k \in \mathbb{N} } $ that satisfy the following: 
    \begin{itemize} 
        \item (Adaptiveness) For each $k\in\mathbb{N}$, let $\mathcal{F}_k=\sigma(\theta_1,\cdots\theta_k)$, and it holds that: 1. $\ene(k)$ is $\mathcal{F}_k$-measurable; 2. $ \{ \varphi_k\}_k $ and $\{ \kappa_k \}_k$ are $\mathcal{F}_{k-1}$-measurable.
        \item (Sub-Gaussianity) $ \{ \theta_k \} $ is conditionally $\sigma$-sub-Gaussian: 
        it holds that 
        \begin{align*} 
            \E \left[ \exp \( \frac{\| \theta_k \|^2}{\sigma^2} \) | \mathcal{F}_{k-1} \right] \le \exp(1). 
        \end{align*} 
        \item (Almost super-martingale condition) The inequality $ \mathcal{E} (k) - \mathcal{E} (k-1) \le a_k \| \theta_k \|^2 + \sqrt{a_k} \< \theta_k, \varphi_k \> + \kappa_k $ holds for each $k\in\mathbb{N}_+$, and 
        \begin{align*}
            \gamma_1 = \sum_{k=1}^\infty a_k < \infty, \quad 
            \gamma_2 = \prod_{k=1}^\infty \( 1 + \sigma^2 a_k \) < \infty. 
        \end{align*} 
        Also, for any $t \in (0, B ]$, 
        \begin{align*} 
            \E \left[ \exp \( t \( \sqrt{a_k} \< \theta_k, \varphi_k \> + \kappa_k \) \) | \mathcal{F}_{k-1} \right] \le \exp \( t a_k \sigma^2 \mathcal{E} (k-1) \).
        \end{align*} 
    \end{itemize} 
    Then it holds that $  \forall \beta \in (0,\frac{1}{2}) $, 
    \begin{align*} 
        &\Pr\left(\sup_{k\geq 0}\ene(k)\geq \frac{\gamma_2}{B} \left(B\ene(0)+\log\frac{1}{\beta}\right)+\left(1+\log\frac{1} {\beta}\right)\sigma^2(1+\sigma^2\gamma_1\gamma_2)\gamma_1\right) \le 2\beta  . 
    \end{align*} 
\end{lemma}

\begin{remark} 
    Lemma \ref{lem:gronwall} establishes a high-probability version of the celebrated convergence result of almost super-martingales \citep{ROBBINS1971233}, under the additional assumption of conditionally sub-Gaussian increments. Specifically, the sequence $\{\mathcal{E}(k)\}_k$ in Lemma \ref{lem:gronwall} is an almost super-martingale in the sense of Robbins-Siegmund. By imposing a sub-Gaussian condition, we derive a large-deviation bound for such almost super-martingales. 
\end{remark} 


\subsection{Highlight of Contributions}

To the best of our knowledge, our Theorem \ref{thm:main} establishes the tightest high-probability convergence result rate for any stopping time $\tau$. This rate of order $\frac{\log \tau}{\sqrt{\tau}}$ provides an improvement over prior works, which, at best, could achieve a high-probability convergence rate of order $\frac{\log^2 \tau}{\sqrt{\tau}}$ in terms of stopping time $\tau$; See Section \ref{sec:prior-detail}. 
Despite an uncountable amount of works on stochastic optimization (See Section \ref{sec:related-works} for a review), our work is the first to directly analyze convergence behavior in terms of the stopping time $\tau$ -- a crucial and practical perspective that aligns with real-world applications. Unlike existing approaches, we focus on adaptive, data-driven stopping times that better reflect real-world needs.


More importantly, we establish a lemma governing the large deviation property of almost super-martingales. This finding could be of broader interest, extending beyond the optimization community.






\textbf{Paper Organization.} The rest of the paper is organized as follows: Section \ref{sec:main} presents the proof of Theorem \ref{thm:main}; Section \ref{sec:prior-detail} discusses the state-of-the-art stopping time convergence prior to our work. Section \ref{sec:related-works} surveys related works before concluding the paper. 

\section{Main Results}
\label{sec:main}

First, we recall the classic Stochastic Gradient Descent with Momentum (SGDM) method \citep{lan2020first}. A variant of this algorithm iterates as follows: 
\begin{align}
\label{eq:SGDM} 
    \begin{split}
        x_{k+1} &= x_k + \frac{k}{k+2}(x_k - x_{k-1}) - \frac{2\sqrt{\eta_k}}{(k+2)\sqrt{k}} \, g(x_k, \xi_k), \quad k \in \mathbb{N_+}, \\
        x_1 &= x_0 \in \R^n,
    \end{split} 
\end{align}
where $\eta_k = \frac{1}{16L^2 \log^2 (k+2)}$ is the learning rate, and $g(x_k, \xi_k)$ is the stochastic gradient at $x_k$. 

With the SGDM algorithm outlined, we can proceed to the proof of Theorem \ref{thm:main}. Before diving into the details, we state some conventions and assumptions that will be used throughout the analysis.



\begin{assumption}\label{assump:convex-smooth}
    The function $f$ is convex and $L$-smooth, meaning that for any $x, y \in \R^n$, the gradient satisfies $\|\nabla f(x) - \nabla f(y)\| \leq L \|x - y\|$. In addition, $f$ attains its minimum $f^* > -\infty$ at some $x^* \in \R^n$. 
\end{assumption}

We denote by $\mathcal{F}_k$ the $\sigma$-algebra generated by all randomness up to the point of reaching $x_k$, i.e., $\mathcal{F}_k = \sigma(\xi_1, \dots, \xi_{k})$. The following assumptions are introduced to facilitate the convergence analysis.

\begin{assumption}\label{assump:sto-gradient}
    For any $k \geq 1$, the stochastic gradient $g(x_k, \xi_k)$ satisfies:
    \begin{itemize}
        \item Unbiasedness: $\E[g(x_k, \xi_k) \mid \mathcal{F}_{k-1}] = \nabla f(x_k)$.
        \item Bounded variance: $\E[\|g(x_k, \xi_k)\|^2 \mid \mathcal{F}_{k-1}] \leq \|\nabla f(x_k)\|^2 + \sigma^2$.
    \end{itemize}
\end{assumption}

\begin{assumption}\label{assump:hp}
    For any $k \geq 1$, the stochastic gradient $g(x_k, \xi_k)$ is conditionally sub-Gaussian: There exists $\sigma > 0$, such that for all $k \in \mathbb{N}_+$, 
    \begin{equation*}
        \E\left[\left.\exp\left(\frac{\|g(x_k, \xi_k) - \nabla f(x_k)\|^2}{\sigma^2}\right)\right|\mathcal{F}_{k-1} \right] \leq \exp(1). 
    \end{equation*}
\end{assumption}



Assumption \ref{assump:sto-gradient} states that the stochastic gradient is unbiased and has bounded variance, which are standard requirements in stochastic optimization.
Assumption \ref{assump:hp} imposes a sub-Gaussian tail condition on the gradient noise, which is a common and mild assumption in high-probability analysis.
With these in place, we are now ready to proceed to the proof of Theorem \ref{thm:main}.

We define the following Lyapunov function 
\begin{align}\label{discrete-energy:alpha}
        \ene(k)=\lnorm x_{k+1}+(k+1)(x_{k+1}-x_k)-x^*\rnorm^2
        +4\sqrt{(k+1)\eta_k}(f(x_{k})-f^*). 
\end{align}
The following result shows a nice property of this Lyapunov function $\ene(k)$. 

\begin{lemma}\label{lem:explicit-decay-difference-alpha}
    Let Assumption \ref{assump:convex-smooth} hold. If the stepsize sequence $\{\eta_k\}$ is monotonically decreasing, then $\{x_k\}$ generated by SGDM (\ref{eq:SGDM}) satisfies
    \begin{align}
        \begin{split}\label{ene-decay-expectation-alpha}
            \ene(k)-\ene(k-1)
            \leq&\frac{4\eta_k}{k}\lnorm g(x_{k},\xi_{k})\rnorm^2-\frac{2}{L}\sqrt{\frac{\eta_k}{k}}\|\nabla f(x_{k})\|^2-2\sqrt{\frac{\eta_k}{k}}(f(x_{k})-f^*)\\
            &+4\sqrt{\frac{\eta_k}{k}}\langle\nabla f(x_{k})-g(x_{k},\xi_{k}),\varphi_{k}\rangle
        \end{split}
    \end{align}
    for any $k\geq1$, where $\varphi_k=k(x_{k}-x_{k-1})+(x_{k}-x^*)$.
\end{lemma} 

The following properties should be noted in this lemma: 
\begin{itemize}
    \item \textbf{(P1)} As per how $\varphi_k$ and $\ene (k)$ are defined, we have 
    \begin{align*} 
        \ene(k) = \lnorm \varphi_{k+1} \rnorm^2 + 4\sqrt{(k+1)\eta_k}(f(x_{k})-f^*) . 
    \end{align*} 
    Since $ \varphi_k $ appears in the last term in \eqref{ene-decay-expectation-alpha}, this property will allow us to further control the right-hand-side of (\ref{ene-decay-expectation-alpha}) by the first term in the Lyapunov function $\ene(k-1)$ defined in (\ref{discrete-energy:alpha}). 
    \item \textbf{(P2)} We know that $ \varphi_k$ is $\mathcal{F}_{k-1} $-measurable. Consequently, after taking the conditional expectation or computing the conditonal MGF, the right-hand-side of (\ref{ene-decay-expectation-alpha}) includes the term $ \| \varphi_k \|^2 $, which can be bounded by $\mathcal{E} (k-1)$ (due to \textbf{(P1)}). This allows for a seamless application of an almost super-martingale argument, which is stated in Lemma \ref{lem:gronwall}. 
\end{itemize} 

Now we prove Lemma \ref{lem:explicit-decay-difference-alpha}. 



\begin{proof}[Proof of Lemma \ref{lem:explicit-decay-difference-alpha}]
    By differnencing the Lyapunov function (\ref{discrete-energy:alpha}), we get
    \begin{align}\label{eq:discrete-diff}
        \begin{split}
            &\ene(k)-\ene(k-1)\\
        \leq&\lnorm x_{k+1}+(k+1)(x_{k+1}-x_{k})-x^*\rnorm^2+4\sqrt{(k+1)\eta_{k}}(f(x_{k})-f^*)\\
        &-\lnorm x_{k}+k(x_{k}-x_{k-1})-x^*\rnorm^2-4\sqrt{k\eta_{k}}(f(x_{k-1})-f^*)\\
        \leq&2\langle 2(x_{k+1}-x_{k})+k(x_{k+1}-2x_{k}+x_{k-1}),
        x_{k+1}+(k+1)(x_{k+1}-x_{k})-x^*\rangle\\
        &-\lnorm 2(x_{k+1}-x_{k})+k(x_{k+1}-2x_{k}+x_{k-1})\rnorm^2+4\sqrt{k\eta_{k}}(f(x_{k})-f(x_{k-1}))\\
        &+2\sqrt{\frac{\eta_k}{k}}(f(x_{k})-f^*),
        \end{split}
    \end{align}
    where the first inequality follows from $\eta_{k}\leq\eta_{k-1}$, and the second inequality follows from $\|a\|^2-\|b\|^2=2\langle a-b,a\rangle-\|a-b\|^2$ and $\sqrt{k+1}-\sqrt{k}\leq\frac{1}{2\sqrt{k}}$. From (\ref{eq:SGDM}) we have
    \begin{align}\label{discrete-alpha}
        \begin{split}
            &2(x_{k+1}-x_{k})+k(x_{k+1}-2x_{k}+x_{k-1})\\
            =&(k+2)(x_{k+1}-x_{k})-k(x_{k}-x_{k-1})
        =-2\sqrt{\frac{\eta_k}{k}}g(x_{k}, \xi_{k}),
        \end{split}
    \end{align}
    and thus
    \begin{align}\label{explicit-decay-mid1}
        \begin{split}
            &\ene(k)-\ene(k-1)\\
            \leq&-4\left\langle\sqrt{\frac{\eta_k}{k}}g(x_{k}, \xi_{k}), x_{k}+(k+2)(x_{k+1}-x_{k})-x^*\right\rangle-\lnorm2\sqrt{\frac{\eta_k}{k}}g(x_{k}, \xi_{k})\rnorm^2\\
            &+4\sqrt{k\eta_{k}}(f(x_{k})-f(x_{k-1}))+2\sqrt{\frac{\eta_k}{k}}(f(x_{k})-f^*)\\
            =&-4\left\langle\sqrt{\frac{\eta_k}{k}} g(x_{k},\xi_{k}),(k+2)(x_{k+1}-x_{k})\right\rangle-\lnorm2\sqrt{\frac{\eta_k}{k}} g(x_{k},\xi_{k})\rnorm^2\\
            &+4\sqrt{k\eta_k}(f(x_{k})-f(x_{k-1}))+4\sqrt{\frac{\eta_k}{k}}(f(x_{k})-f^*-\langle g(x_{k},\xi_{k}), x_{k}-x^*\rangle)\\
            &-2\sqrt{\frac{\eta_k}{k}}(f(x_{k})-f^*).
        \end{split}
    \end{align}
    Then we note that the convexity and $L$-smoothness of $f$ gives
    \begin{equation}\label{explicit-decay-convex1}
        f^*-f(x_{k})-\langle\nabla f(x_{k}), x^*-x_{k}\rangle\geq\frac{1}{2L}\|\nabla f(x_{k})\|^2,
    \end{equation}
    and the convexity of $f$ shows that 
    \begin{equation}\label{explicit-decay-convex2}
        f(x_{k})-f(x_{k-1})\leq\langle\nabla f(x_{k}), x_{k}-x_{k-1}\rangle.
    \end{equation}
    Plugging (\ref{explicit-decay-convex1}) and (\ref{explicit-decay-convex2}) into (\ref{explicit-decay-mid1}), we have
    \begin{align*}
        &\ene(k)-\ene(k-1)\\
        \leq&4\sqrt{\frac{\eta_k}{k}}\langle g(x_k,\xi_k),-(k+2)x_{k+1}+(2k+2)x_k-kx_{k-1}\rangle\\
        &+4\sqrt{\frac{\eta_k}{k}}\langle\nabla f(x_{k})-g(x_{k},\xi_{k}),k(x_k-x_{k-1})\rangle\\
        &-\lnorm2\sqrt{\frac{\eta_k}{k}} g(x_{k},\xi_{k})\rnorm^2-\frac{2}{L}\sqrt{\frac{\eta_k}{k}}\|\nabla f(x_{k})\|^2
        +4\sqrt{\frac{\eta_k}{k}}\langle\nabla f(x_{k})-g(x_{k},\xi_{k}),x_{k}-x^*\rangle\\
        &-2\sqrt{\frac{\eta_k}{k}}(f(x_{k})-f^*).\nonumber
    \end{align*}
    Finally, we substitute (\ref{discrete-alpha}) to the RHS of the above inequality, and reorder the terms to conclude
    \begin{align*}
        \ene(k)-\ene(k-1) 
        \leq&\frac{4\eta_k}{k}\lnorm g(x_{k},\xi_{k})\rnorm^2-\frac{2}{L}\sqrt{\frac{\eta_k}{k}}\|\nabla f(x_{k})\|^2-2\sqrt{\frac{\eta_k}{k}}(f(x_{k})-f^*)\\
        &+4\sqrt{\frac{\eta_k}{k}}\langle\nabla f(x_{k})-g(x_{k},\xi_{k}), \varphi_k\rangle,
    \end{align*}
    where $\varphi_k=k(x_{k}-x_{k-1})+(x_{k}-x^*)$.
\end{proof}



\subsection{An Almost Super-martingale Analysis} 

Now we present the high-probability convergence result for Stochastic Gradient Descent with Momentum (SGDM) in terms of stopping times. Our analysis relies on the dynamical properties captured by the discrete Lyapunov function $\ene(k)$ and Lemma \ref{lem:explicit-decay-difference-alpha}.


To provide intuition, the Lyapunov function $\ene(k)$ and Lemma \ref{lem:explicit-decay-difference-alpha} describe the dynamical behavior of SGDM. Specifically, by utilizing \textbf{(P1)} and \textbf{(P2)}, Lemma \ref{lem:explicit-decay-difference-alpha} yields:
\begin{align}\label{eq:descent-property}
    \begin{split}
        \ene(k) - \ene(k-1) 
        &\leq \frac{8\eta_k}{k}\|\theta_k\|^2 + \frac{8\eta_k}{k}\|\nabla f(x_k)\|^2 - \frac{2}{L}\sqrt{\frac{\eta_k}{k}}\|\nabla f(x_k)\|^2 \\
        &\quad + 4\sqrt{\frac{\eta_k}{k}}\langle\theta_k, \varphi_k\rangle \\
        &\leq a_k\|\theta_k\|^2 + \sqrt{a_k}\langle\theta_k, \varphi_k\rangle,
    \end{split}
\end{align}
where
\begin{align}\label{eq:descent-property-def}
    \varphi_k = k(x_k - x_{k-1}) + (x_k - x^*),\; \theta_k = \nabla f(x_k) - g(x_k, \xi_k),\text{ and }a_k = \frac{16\eta_k}{k}.
\end{align}
The second inequality follows from $\eta_k \leq \frac{k}{16L^2}$ for $k \geq 1$.

The inequality (\ref{eq:descent-property}) decomposes the increment of $\ene(k)$ into two terms $ a_k \| \theta_k \|^2 $ and $ \sqrt{a_k} \< \theta_k, \varphi_k \> $. 
To bound $ \ene (k) $, we bound the conditional MGF of $ \ene (k) - \ene (k-1) $ in terms of $ \ene (k-1) $, which leads to an almost super-martingale recurrence. Next we present a key lemma that leads to the proof of Theorem \ref{thm:main}.

\begin{proof}[Proof of Lemma \ref{lem:gronwall}]
    We denote $S(k)=\sum_{l=1}^ka_l\|\theta_l\|^2$ and $M(k)=\ene(k)-S(k)$. This gives
    \begin{align}\label{eq:diff-of-M}
        M(k)-M(k-1)=(\ene(k)-\ene(k-1))-(S(k)-S(k-1))
        \leq\sqrt{a_k}\langle\theta_k,\varphi_{k}\rangle + \kappa_k.
    \end{align}

    For the moment-generating function of $M(k)$, (\ref{eq:diff-of-M}) gives, for any $t \in (0,B]$, 
    \begin{align}
        \begin{split}
            \mathbb{E}[\exp(t M(k))|\mathcal{F}_{k-1}] 
            \leq&\;  
            \exp(t M(k-1) ) \mathbb{E}[\exp (t \sqrt{a_k} \langle \theta_k,\varphi_k\rangle + t \kappa_k ) |\mathcal{F}_{k-1}] \\ 
            \le& \;  
            \exp(t M(k-1) ) \exp (t a_k \mathcal{E} (k-1) )  \\ 
            \le&\; 
            \exp( (1 + a_k \sigma^2 ) t M(k-1) + a_k \sigma^2 t S (k-1) ) 
        \end{split}
        \label{eq:mgf-decomposition}
    \end{align}
    where the second inequality follows from the almost super-martingale condition (in the lemma statement). 

    We write 
    \begin{align*} 
        N^t(k)=\exp\left(\prod_{l=k+1}^{\infty}(1+a_l \sigma^2) t M(k)-\sum_{l=1}^{k}a_l\sigma^2\gamma_2 tS(l-1)\right). 
    \end{align*} 
    Specifically, this definition gives $N^t(0)=\exp\left(\prod_{l=1}^\infty(1+a_l\sigma^2)tM(0)\right)=\exp(\gamma_2t\ene(0))$. The conditional expectation of $N^t(k)$ for $t\leq\frac{B}{\gamma_2}$ satisfies 
    \begin{align} 
        &\E[N^t(k)|\mathcal{F}_{k-1}]\nonumber\\
        =&\E\left[\left.\exp\left(\prod_{l=k+1}^{\infty}(1+a_l \sigma^2) t M(k)-\sum_{l=1}^{k}a_l\sigma^2\gamma_2 tS(l-1)\right)\right|\mathcal{F}_{k-1}\right]\nonumber\\
        =&\E\left[\left.\exp\left(\prod_{l=k+1}^{\infty}(1+a_l \sigma^2) t M(k)\right)\right|\mathcal{F}_{k-1}\right]\cdot\exp\left(-\sum_{l=1}^{k}a_l\sigma^2\gamma_2 tS(l-1)\right)\nonumber\\
        \leq&\exp\left(\prod_{l=k}^{\infty}(1+a_l \sigma^2) t M(k-1)+a_k\sigma^2\prod_{l=k+1}^{\infty}(1+a_l \sigma^2)tS(k-1)\right.\nonumber\\
        &\phantom{\exp\bigg(}\left.-\sum_{l=1}^{k}a_l\sigma^2\gamma_2 tS(l-1)\right)\label{eq:Ntk}\\
        =&N^t(k-1)\cdot\exp\left(a_k\sigma^2\prod_{l=k+1}^{\infty}(1+a_l \sigma^2)tS(k-1)-a_k\sigma^2\gamma_2 tS(k-1)\right)\nonumber\\
        \leq& N^t(k-1),\nonumber 
    \end{align} 
     where (\ref{eq:Ntk}) follows from  (\ref{eq:mgf-decomposition}). Therefore, $\{N^t(k)\}_{k=0}^\infty$ is a supermartingale for $0<t\leq\frac{B}{\gamma_2}$. Then applying Ville's inequality \citep{ville1939etude} to $\{N^t(k)\}_{k=0}^\infty$ gives 
    \begin{align}\label{eq:doob}
        \Pr\left(\sup_{k\geq0}N^t(k)\geq\exp(\alpha t)\right)\leq\exp(-\alpha t)\E[N^t(0)]=\exp(-\alpha t+\gamma_2t\ene(0)). 
    \end{align} 
    Since $\gamma_1 S(k) \geq \sum_{l=1}^k a_l S(l-1)$, we have 
    \begin{align*} 
        &\; \exp\left(\frac{B M(k)}{\gamma_2} - B \sigma^2\gamma_1S(k)\right) \\ 
        \leq&\;  
        \exp \left(\frac{B\prod_{l=k+1}^\infty(1+a_l\sigma^2)M(k)}{\gamma_2} - B \sum_{l=1}^ka_l\sigma^2 S(l-1)\right) 
        = 
        N^{\frac{B}{\gamma_2}}(k), 
    \end{align*} 
    and thus
    \begin{align}\label{eq:anytime-final1}
        \begin{split}
            &\Pr\left(\sup_{k\geq0}\left\{M(k)-\sigma^2\gamma_1\gamma_2S(k)\right\}\geq \frac{\gamma_2}{B} \left(B\ene(0)+\log\frac{1}{\beta}\right)\right)\\ 
            =&
            \Pr\left(\sup_{k\geq0}\left\{\exp\left(\frac{B M(k)}{\gamma_2} - B \sigma^2\gamma_1S(k)\right)\right\}\geq\exp\left(B\ene(0)+\log\frac{1}{\beta}\right)\right)\\
            \leq&
            \Pr\left(\sup_{k\geq0}N^{\frac{B}{\gamma_2}}(k)\geq\exp\left(B\ene(0)+\log\frac{1}{\beta}\right)\right)\leq\beta,
        \end{split}
    \end{align}
    where the last inequality follows from (\ref{eq:doob}) with $t=\frac{B}{\gamma_2}$ and $\alpha=\frac{\gamma_2}{B} \left(B\ene(0)+\log\frac{1}{\beta}\right)$.

    On the other hand, the sub-Gaussianity condition and Lemma \ref{lemma:large-dev-1} gives, for any $\beta \in (0,\frac{1}{2})$ and any $k \in \mathbb{N}$, 
    \begin{align*}
            &\Pr\left((1+\sigma^2\gamma_1\gamma_2)S(k)\geq\left(1+\log\frac{1}{\beta}\right)\sigma^2(1+\sigma^2\gamma_1\gamma_2)\gamma_1\right)\\
        \leq&\Pr\left((1+\sigma^2\gamma_1\gamma_2)S(k)\geq\left(1+\log\frac{1}{\beta}\right)\sigma^2(1+\sigma^2\gamma_1\gamma_2)\sum_{l=1}^ka_l\right)\leq\beta.
    \end{align*}
    Since $\{S(k)\}_{k=0}^\infty$ is monotonically increasing, we know from continuity of probability measure, for any $\beta \in (0,\frac{1}{2})$, 
    \begin{align}\label{eq:anytime-final2}
        &\Pr\left(\sup_{k\geq0}\left\{(1+\sigma^2\gamma_1\gamma_2)S(k)\right\}\geq\left(1+\log\frac{1}{\beta}\right)\sigma^2(1+\sigma^2\gamma_1\gamma_2)\gamma_1\right)\leq\beta.
    \end{align}

    Combining (\ref{eq:anytime-final1}) and (\ref{eq:anytime-final2}), we derive
    \begin{align*}
        &\Pr\left(\sup_{k\geq 0}\ene(k)\geq \frac{ \gamma_2 }{B} \left(B\ene(0)+\log\frac{1}{\beta}\right)+\left(1+\log\frac{1} {\beta}\right)\sigma^2(1+\sigma^2\gamma_1\gamma_2)\gamma_1\right)\\ 
        \leq&  \Pr\left(\sup_{k\geq0}\left\{M(k)-\sigma^2\gamma_1\gamma_2S(k)\right\}\geq \frac{\gamma_2}{B} \left(B\ene(0)+\log\frac{1}{\beta}\right)\right)\\ 
        &+\Pr\left(\sup_{k\geq0}\left\{(1+\sigma^2\gamma_1\gamma_2)S(k)\right\}\geq\left(1+\log\frac{1} {\beta}\right)\sigma^2(1+\sigma^2\gamma_1\gamma_2)\gamma_1\right)\\ 
        \leq&2\beta, 
    \end{align*}
    for any $\beta \in (0,\frac{1}{2})$. 
\end{proof}

\begin{proof}[Proof of Theorem \ref{thm:main}]

    We verify that the Lyapunov function $\mathcal{E} (k)$, the sequences $ \{\theta_k\} $, $\{\varphi_k\}$ and $\{a_k\}$ in (\ref{eq:descent-property-def}); the constant $\sigma$ in Assumption \ref{assump:hp}, $\kappa_k\equiv0$ and $B=1$ satisfy the conditions in Lemma \ref{lem:gronwall}. 

    Firstly, the adaptiveness condition is easily verified from the definitions of $\{\ene(k)\}_k$, $\{\varphi_k\}_k$ and $\{\kappa_k\}_k$, and the sub-Gaussianity condition follows from Assumption \ref{assump:hp}. The finiteness condition follows from (\ref{eq:descent-property}) and (\ref{eq:descent-property-def}). Definition of the Lyapunov function (\ref{discrete-energy:alpha}) yields that $\|\varphi_k\|^2\leq\ene(k-1)$, and thus Lemma \ref{lemma:large-dev-3} and Assumption \ref{assump:hp} gives $ \forall t \in (0,1] $, 
    \begin{align*}
        \mathbb{E}[\exp (t \sqrt{a_k} \langle \theta_k,\varphi_k\rangle) |\mathcal{F}_{k-1}]
        \leq
        \exp(a_k\sigma^2t^2\ene(k-1))
        \le 
        \exp(a_k\sigma^2t \ene(k-1)) ,  
    \end{align*}
    which verifies the almost super-martingale condition.
    

    Therefore, we can apply Lemma \ref{lem:gronwall} (with $B = 1$) to obtain: $\forall \beta \in (0,\frac{1}{2})$, 
    \begin{align*}
        \Pr\left(\sup_{k\geq 0}\ene(k)\geq \gamma_2 \left(\ene(0)+\log\frac{1}{\beta}\right)+\left(1+\log\frac{1} {\beta}\right)\sigma^2(1+\sigma^2\gamma_1\gamma_2)\gamma_1\right) \le 2 \beta . 
    \end{align*}
    Since $4\sqrt{(k+1)\eta_k}(f(x_k)-f^*)\leq\ene(k)$, we conclude that
    \begin{align*}
        \Pr\left(f(x_k)-f^*\leq\frac{\left(C_1+C_2\log\frac{1}{\beta}\right)\log(k+2)}{\sqrt{k+1}},\;\text{for all $k\geq0$}\right)\geq1-2\beta,
    \end{align*}
    where $C_1=L\gamma_2\ene(0)+L\sigma^2(1+\sigma^2\gamma_1\gamma_2)\gamma_1$, and $C_2=L\gamma_2+L\sigma^2(1+\sigma^2\gamma_1\gamma_2)\gamma_1$.
\end{proof}


Building on Theorem \ref{thm:main}'s analysis, we proceed to prove Proposition \ref{prop}.

\begin{proof}[Proof of Proposition \ref{prop}]
    As in the proof of Theorem \ref{thm:main}, we pick $ \eta_k $ properly so that $a_k=\frac{16\eta_k}{k} = \frac{1}{L^2 C_0' k\log^{(1+\varepsilon)}(k+2)}$, where $C_0' \ge 1$ is some constant (independent of $ \varepsilon, \beta, k $). 
    It is straightforward to verify that when $C_0'$ is large enough (e.g., $C_0' \ge 100$), it holds that 
    \begin{align*} 
        \gamma_1 := \sum_{k=1}^\infty a_k \le \zeta (1 + \varepsilon) , 
    \end{align*}
    and 
    \begin{align*}
        \gamma_2 =&\; \prod_{k=1}^\infty(1+a_k\sigma^2) =  
        \exp \log \( \prod_{k=1}^\infty(1+a_k\sigma^2) \) \\ 
        \le& \;  
        \exp \( \sum_{k=1}^\infty a_k\sigma^2 \) \le \exp \( \sigma^2 \zeta ( 1 + \varepsilon ) \) . 
    \end{align*} 


    The same argument as the proof of Theorem \ref{thm:main} yields that
    \begin{align*}
        \Pr\left(\sup_{k\geq 0}\ene(k)\geq \gamma_2\left(\ene(0)+\log\frac{1}{\beta}\right)+\left(1+\log\frac{1}{\beta}\right)\sigma^2(1+\sigma^2\gamma_1\gamma_2)\gamma_1\right)\leq2\beta.
    \end{align*}
    Since $4\sqrt{(k+1)\eta_k}(f(x_k)-f^*)\leq\ene(k)$, we conclude that
    \begin{align*}
        \Pr\left(\forall k\geq0,\;f(x_k)-f^*\leq\frac{\left(C_1+C_2\log\frac{1}{\beta}\right)\log^\frac{1+\varepsilon}{2}(k+2)}{\sqrt{k+1}}\right)\geq1-2\beta,
    \end{align*}
    where $C_1=L\gamma_2\ene(0)+L\sigma^2(1+\sigma^2\gamma_1\gamma_2)\gamma_1 \le C_0 \cdot \( \exp \( \sigma^2 \zeta ( 1 + \varepsilon ) \) ( \zeta (1 + \varepsilon) )^2 \) $, and $C_2=L\gamma_2+L\sigma^2(1+\sigma^2\gamma_1\gamma_2)\gamma_1 \le C_0 \cdot \( \exp \( \sigma^2 \zeta ( 1 + \varepsilon ) \) ( \zeta (1 + \varepsilon) )^2 \) $, for some constant $C_0$ that is independent of $k,\beta,\varepsilon$. 
\end{proof}




In the above analysis, we have established that with high probability, $ f (x_k) -f^* \lesssim \frac{\log k}{\sqrt{k}} $ \textbf{simultaneously} for all $k$. Now we investigate how such statement relates to the stopping time convergence rate. 
The following lemma makes this connection precise.  

\begin{lemma}\label{lemma:iff}
    Let $U(\cdot,\cdot)$ be a $\R^+\times\N^+\to\R^+$ function, $\beta$ be a positive constant, and $\{x_t\}_{t\in\N}$ be the sequence generated by a stochastic optimization method applied on $f$. Then
    \begin{align}\label{eq:stopping-time-convergence-2}
        \Pr\left(f(x_\tau)-f^*\leq U(\beta,\tau)\right)\geq1-\beta, \text{ for any $\{x_t\}$-stopping time $\tau$}
    \end{align}
    if and only if
    \begin{align}\label{eq:stopping-time-convergence-3}
        \Pr\left(\bigcap_{k\in\N^+}\left\{f(x_k)-f^*\leq U(\beta,k)\right\}\right)\geq1-\beta.
    \end{align}
\end{lemma}

\begin{proof}[Proof of Lemma \ref{lemma:iff}]
    To prove the equivalence of (\ref{eq:stopping-time-convergence-2}) and (\ref{eq:stopping-time-convergence-3}), we first observe that (\ref{eq:stopping-time-convergence-3}) trivially implies (\ref{eq:stopping-time-convergence-2}). For the converse, in order to get a contradiction, we suppose that 
    \begin{align}\label{eq:stopping-time-convergence-assumption}
        \Pr\left(\bigcap_{k\in\N^+}\left\{f(x_k)-f^*\leq U(\beta,k)\right\}\right)<1-\beta.
    \end{align}
    For each $k$, we define event $E_k$ as $E_k=\left\{f(x_k)-f^*>U(\beta,k)\right\}\bigcap\{f(x_l)-f^*\leq U(\beta,l),\;\forall l<k\}$. Then (\ref{eq:stopping-time-convergence-assumption}) yields that $\Pr\left(\bigcup_{k\in\N^+}E_k\right)>\beta$, and thus there exists a $k_0$ such that $\Pr\left(\bigcup_{k\leq k_0}E_k\right)>\beta$. Then we define the random variable $\tau$ as
    \begin{align*}
        \tau(\omega)=\left\{
            \begin{aligned}
                &k,\quad\text{if $\omega\in E_k$ for $k\leq k_0$},\\
                &k_0 + 1,\quad\text{otherwise.}
            \end{aligned}
        \right.
    \end{align*}
    It is easy to verify that $\tau$ is well-defined and is a $\{x_k\}$-stopping time. From the definition we have
    \begin{align*}
        \Pr(f(x_\tau)-f^*\leq U(\beta,\tau))\leq1-\Pr\left(\bigcup_{k\leq k_0}E_k\right)<1-\beta.
    \end{align*}
    This is a contradiction to (\ref{eq:stopping-time-convergence-2}), and thus we finish the proof.
\end{proof}


\section{Stopping time convergence rate based on existing results} 
\label{sec:prior-detail}

To better contextualize our contribution, we present a stopping time convergence rate analysis based on existing results. 
For the stochastic optimization problem with the conditions and assumptions stated in Theorem \ref{thm:main}, \cite{liu2023high,liu2023revisiting} proved that the sequence $\{x_k\}$ generated by the stochastic first-order method proposed by \cite{lan2020first} and the classic stochastic gradient descent satisfy: 
\begin{align} 
    \label{eq:up-liu} 
    \Pr\left(f(x_k)-f^*\lesssim\frac{1} {\sqrt{k}}\left(\frac{1}{\eta}+\eta\cdot\log\frac{1}{\delta}\cdot\log k\right)\right)\geq1-\delta, \quad \forall k \in \mathbb{N}_+, \; \forall \delta \in (0,1), 
\end{align} 
where $\eta$ is a parameter. Eq. \eqref{eq:up-liu} is stated in Corollary B.3 in \cite{liu2023high} and Corollary C.2 in \cite{liu2023revisiting}, which is the cornerstone for Theorem 3.6 in \cite{liu2023high} and Theorem 3.3 in \cite{liu2023revisiting}. 
We can pick any $\beta \in (0,1)$, and choose $\delta=\frac{6\beta}{\pi^2k^2}$ to obtain 
\begin{align*} 
    \Pr\left(f(x_k)-f^*\lesssim\frac{1}{\sqrt{k}}\( \frac{1}{\eta} + \eta \cdot \log\frac{\pi^2 k^2}{6\beta}\log k \) \right)\geq1-\frac{6\beta}{\pi^2k^2}, \quad \forall k \in \mathbb{N}_+, \; \forall \beta \in (0,1) . 
\end{align*} 
By taking a union bound over $k$, we conclude that, $ \forall \beta \in (0,1) $, 
\begin{align*} 
    \Pr\left( \forall k \in \mathbb{N}_+, \; f(x_k )-f^*\lesssim\frac{1}{\sqrt{k}} \( \frac{1}{\eta} + \eta \cdot \log\frac{k}{\beta}\log k \) \right)\geq1- \sum_{k=1}^\infty \frac{6\beta}{\pi^2k^2} = 1- \beta . 
\end{align*} 
By Lemma \ref{lemma:iff}, the above high-probability result is equivalent to a high probability convergence in terms of any stopping time $\tau$, 
\begin{align}
    \Pr\left(  f(x_\tau )-f^*\lesssim\frac{1}{\sqrt{\tau}} \( \frac{1}{\eta} + \eta \cdot \log\frac{\tau}{\beta}\log \tau \) \right)
    \geq 1 - \beta, \quad  \forall \beta \in (0,1) , 
    \label{eq:prior-union} 
\end{align}
where $\eta$ needs to be independent of the stopping time $\tau$; Otherwise, the stopping time would be known a priori, which is absurd.
Consequently, the resulting $\log^2 \tau$ term in (\ref{eq:prior-union}) is a barrier in existing results. 




\section{Prior Arts}
\label{sec:related-works}

Although no existing works directly focus on convergence analysis in terms of stopping time, numerous ingenious researchers have contributed to the field of stochastic optimization.

Stochastic optimization has a long history, with its origins dating back to at least \cite{robbins1951stochastic,kiefer1952stochastic}. In recent years, the field has experienced significant growth, driven by the widespread use of mini-batch training in machine learning \cite{GoodBengCour16}, which induces stochasticity. Researchers have approached the problem of stochastic optimization from various perspectives. For instance, \cite{johnson2013accelerating}, \cite{allen2016variance}, \cite{allen2017katyusha}, \cite{allen2018katyusha}, and \cite{ge2019stabilized} investigated algorithms that only use full-batch training sporadically for finite-sum optimization problems. This class of algorithms are commonly known as Stochastic Variance Reduced Gradient (SVRG). 
\cite{li2019convergence}, \cite{ward2020adagrad}, and \cite{faw2022power} provided theoretical results for stochastic gradient methods with adaptive stepsizes.
\cite{zhang2020adaptive} demonstrated, both theoretically and empirically, that the heavy-tailed nature of gradient noise contributes to the advantage of adaptive gradient methods over SGD.
\cite{gadat2018stochastic, kidambi2018insufficiency, gitman2019understanding} studied the role of momentum in stochastic gradient methods. 
\cite{karimi2016linear} and \cite{madden2020high} proved the convergence of stochastic gradient methods under the PL condition. \cite{sebbouh2021almost} showed the almost sure convergence of SGD and the stochastic heavy-ball method.
The convergence of stochastic gradient methods in expectation is also well studied. \cite{shamir2013stochastic} studied the last-iterate convergence of SGD in expectation for non-smooth objective functions, and provided an optimal averaging scheme. \cite{yan2018unified} presented a unified convergence analysis for stochastic momentum methods. \cite{assran2020convergence} studied the convergence of Nesterov's accelerated gradient method in stochastic settings. \cite{liu2020improved} also provided an improved convergence analysis in expectation of a momentumized SGD, and proved the benefit of using the multistage strategy.

For theoretical analysis of stochastic optimization algorithms, an important topic is to develop convergence guarantees that hold with high probability \citep{nemirovski2009robust, lan2012optimal, nazin2019algorithms, gorbunov2020stochastic, liu2023high}. This is partially because the high probability analysis better reflects the performance of a single run of the algorithm. Along this line, the early work \cite{nemirovski2009robust} established robust stochastic approximation convergence. Later \cite{lan2012optimal} proposed accelerated stochastic methods with universal optimality for non-smooth/smooth stochastic problems; 
\cite{ghadimi2013stochastic} improved large-deviation properties via a two-phase method.
Recent studies \citep[e.g.,][]{nazin2019algorithms,gorbunov2020stochastic,cutkosky2021high,li2022high} relaxed the light-tailed noise assumption, handling heavy-tailed noise via clipping techniques. 
Adaptive stepsize methods were also studied \citep{li2020high,kavis2022high}. 
The last-iterate convergence behavior has also witnessed remarkable advances \citep{harvey2019tight,jain2021making}, and \cite{liu2023revisiting} unified the last-iterate analysis.

\section{Conclusion}

In this paper, we study the convergence behavior of stochastic optimization algorithms, in terms of stopping times. By improving the high-probability convergence rate from the order of $\frac{\log^2 \tau}{\sqrt{\tau}}$ to the order of $\frac{\log \tau}{\sqrt{\tau}}$, we have broken through a logarithmic barrier. 

Our analysis presents a new large deviation lemma for almost super-martingales (Lemma \ref{lem:gronwall}). This argument can potentially be applied  to other problems in stopping time convergence analysis of other stochastic approximation problems. 



\section*{Acknowledgement}


The authors extend their appreciation to Zijian Liu for his thoughtful comments and helpful discussions regarding prior research in this field.

\appendix


\section{Auxiliary lemmas}\label{app:proof-thm-hp}

The proof of Theorem \ref{thm:main} relies on the following two auxiliary lemmas concerning standard properties for stochastic processes. Such properties can be found in texts on probability and statistics \cite{10.1093/acprof:oso/9780199535255.001.0001,vershynin2018high,wainwright2019high}, and have appeared in classic works \cite{lan2012validation, devolder2011stochastic} in the optimization community. Here we include the proofs for completeness. 

\begin{lemma}\label{lemma:large-dev-3}
    Let $\theta_1,\cdots,\theta_k$ be a sequence of i.i.d. random variables, $\Gamma_l=\Gamma\left(\theta_{[l]}\right)$ and $\Delta_l=\Delta\left(\theta_{[l]}\right)$ be deterministic functions of $\theta_{[l]}=(\theta_1,\cdots,\theta_l)$, and $c_1,\cdots,c_k$ be a sequence of positive numbers such that:
    \begin{enumerate}
        \item $\E[\Gamma_l|\theta_{[l-1]}]=0$,\label{lem:c1}
        \item $|\Gamma_l|\leq c_l\Delta_l$,\label{lem:c2}
        \item $\E\left[\left.\exp\left(\frac{\Delta_l^2}{\sigma^2}\right)\right|\theta_{[l-1]}\right]\leq\exp(1)$,\label{lem:c3}
    \end{enumerate}
    hold for each $l\leq k$. Then, for any $\lambda\in\R$ and $l\geq1$, we have 
    \begin{align*}
        \E\lbm\left.\exp\lb\frac{\lambda\Gamma_l}{c_l\sigma}\rb\right|\theta_{[l-1]}\rbm\leq\exp\lb\frac{3\lambda^2}{4}\rb.
    \end{align*}
\end{lemma}

\begin{proof}[Proof of Lemma \ref{lemma:large-dev-3}]
    Since $\exp(x)\leq x+\exp\lb\frac{9x^2}{16}\rb$ for any $x$, we have
    \begin{align*}
        \E\lbm\left.\exp\lb\frac{\lambda\Gamma_l}{c_l\sigma}\rb\right|\theta_{[l-1]}\rbm\leq& \E\lbm\left.\frac{\lambda\Gamma_l}{c_l\sigma}\right|\theta_{[l-1]}\rbm+\E\lbm\left.\exp\lb\frac{9\lambda^2\Gamma_l^2}{16c_l^2\sigma^2}\rb\right|\theta_{[l-1]}\rbm\\
        \leq& \E\lbm\left.\exp\lb\frac{9\lambda^2\Delta_l^2}{16\sigma^2}\rb\right|\theta_{[l-1]}\rbm,
    \end{align*}
    for any $\lambda$, where the second inequality uses conditions \ref{lem:c1} and \ref{lem:c2}. From the concavity of $f(x)=x^p$ for $0<p\leq1$, we further have
    \begin{align}\label{eq:lem-con1}
        \E\lbm\left.\exp\lb\frac{\lambda\Gamma_l}{c_l\sigma}\rb\right|\theta_{[l-1]}\rbm
        \leq
        \lb\E\lbm\left.\exp\lb\frac{\Delta_l^2}{\sigma^2}\rb\right|\theta_{[l-1]}\rbm\rb^\frac{9\lambda^2}{16}
        \leq
        \exp\lb\frac{9\lambda^2}{16}\rb,
    \end{align}
    for any $0<\lambda\leq\frac{4}{3}$, where the second inequality uses condition \ref{lem:c3}. On the other hand, since $\lambda x\leq\frac{3}{8}\lambda^2+\frac{2}{3}x^2$ for any $\lambda$ and $x$, we have
    \begin{align}
        \begin{split}\label{eq:lem-con2}
            \E\lbm\left.\exp\lb\frac{\lambda\Gamma_l}{c_l\sigma}\rb\right|\theta_{[l-1]}\rbm
            \leq&
            \exp\lb\frac{3\lambda^2}{8}\rb
            \E\lbm\left.\exp\lb\frac{2\Gamma_l^2}{3c_l^2\sigma^2}\rb\right|\theta_{[l-1]}\rbm\\
            \leq&
            \exp\lb\frac{3\lambda^2}{8}\rb
            \E\lbm\left.\exp\lb\frac{2\Delta_l^2}{3\sigma^2}\rb\right|\theta_{[l-1]}\rbm\\
            \leq&\exp\lb\frac{3\lambda^2}{8}+\frac{2}{3}\rb,
        \end{split}
    \end{align}
    for any $\lambda$, where the third inequality follows from the concavity of $f(x)=x^\frac{2}{3}$ and condition \ref{lem:c3}. Combining (\ref{eq:lem-con1}) with (\ref{eq:lem-con2}) yields, for any $\lambda$,
    \begin{align*}
        \E\lbm\left.\exp\lb\frac{\lambda\Gamma_l}{c_l\sigma}\rb\right|\theta_{[l-1]}\rbm\leq\exp\lb\frac{3\lambda^2}{4}\rb,
    \end{align*}
    and thus
    \begin{align*}
        \E\lbm\left.\exp\lb\lambda\Gamma_l\rb\right|\theta_{[l-1]}\rbm\leq\exp\lb\frac{3\lambda^2c_l^2\sigma^2}{4}\rb.
    \end{align*}
\end{proof}

\begin{lemma}\label{lemma:large-dev-1}
    Let $\theta_1,\cdots,\theta_k$ be a sequence of i.i.d. random variables, $\Phi_l=\Phi\left(\theta_{[l]}\right)$ be deterministic functions of $\theta_{[l]}=(\theta_1,\cdots,\theta_l)$, and $c_1,\cdots,c_k$ be a sequence of positive numbers. If the inequality
    \begin{align*}
        \E\left[\left.\exp\left(\frac{\Phi_l^2}{\sigma^2}\right)\right|\theta_{[l-1]}\right]\leq\exp(1)
    \end{align*}
    holds for each $l\leq k$, then for any $\Omega\geq0$, we have
    \begin{align*}
        \Pr\left(\sum_{l=1}^kc_l\Phi_l^2\geq(1+\Omega)\sum_{l=1}^kc_l\sigma^2\right)\leq\exp(-\Omega).
    \end{align*}
\end{lemma}

\begin{proof}[Proof of Lemma \ref{lemma:large-dev-1}]
    Since $f(x)=e^x$ is convex, we have
    \begin{align*}
        \E\lbm\exp\lb\frac{\sum_{l=1}^kc_l\Phi_l^2}{\sum_{l=1}^kc_l\sigma^2}\rb\rbm=\E\lbm\exp\lb\frac{\sum_{l=1}^kc_l\sigma^2\frac{\Phi_l^2}{\sigma^2}}{\sum_{l=1}^kc_l\sigma^2}\rb\rbm
        \leq\frac{\sum_{l=1}^kc_l\sigma^2\E\lbm\exp\lb\frac{\Phi_l^2}{\sigma^2}\rb\rbm}{\sum_{l=1}^kc_l\sigma^2}.
    \end{align*}
    For each $l$, we have $\E\lbm\exp\lb\frac{\Phi_l^2}{\sigma^2}\rb\rbm=\E\lbm\E\lbm\left.\exp\lb\frac{\Phi_l^2}{\sigma^2}\rb\right|\theta_{[l-1]}\rbm\rbm\leq\exp(1)$, and thus
    \begin{align*}
        \E\lbm\exp\lb\frac{\sum_{l=1}^kc_l\Phi_l^2}{\sum_{l=1}^kc_l\sigma^2}\rb\rbm\leq\exp(1).
    \end{align*}
    Then Markov inequality gives, for all $\delta>0$,
    \begin{align*}
        \Pr\lb\exp\lb\frac{\sum_{l=1}^kc_l\Phi_l^2}{\sum_{l=1}^kc_l\sigma^2}\rb\geq\delta\rb\leq\frac{\exp(1)}{\delta}.
    \end{align*}
    Therefore, for any $\Omega\geq0$, by letting $\delta=\exp(1+\Omega)$ and using the monotonicity of $f(x)=e^x$, we conclude that
    \begin{equation*}
        \Pr\left(\sum_{l=1}^kc_l\Phi_l^2\geq(1+\Omega)\sum_{l=1}^kc_l\sigma^2\right)\leq\frac{\exp(1)}{\exp(1+\Omega)}=\exp(-\Omega).
    \end{equation*}
    Hence, we arrive at the desired concentration bound. 
\end{proof} 

\bibliographystyle{plain} 
\bibliography{references}

\begin{thebibliography}{10}

\bibitem{allen2017katyusha}
Zeyuan Allen-Zhu.
\newblock Katyusha: The first direct acceleration of stochastic gradient
  methods.
\newblock In {\em Proceedings of the 49th Annual ACM SIGACT Symposium on Theory
  of Computing}, pages 1200--1205, 2017.

\bibitem{allen2018katyusha}
Zeyuan Allen-Zhu.
\newblock Katyusha {X}: Simple momentum method for stochastic sum-of-nonconvex
  optimization.
\newblock In {\em International Conference on Machine Learning}, pages
  179--185. PMLR, 2018.

\bibitem{allen2016variance}
Zeyuan Allen-Zhu and Elad Hazan.
\newblock Variance reduction for faster non-convex optimization.
\newblock In {\em International conference on machine learning}, pages
  699--707. PMLR, 2016.

\bibitem{assran2020convergence}
Mahmoud Assran and Mike Rabbat.
\newblock On the convergence of nesterov’s accelerated gradient method in
  stochastic settings.
\newblock In {\em International Conference on Machine Learning}, pages
  410--420. PMLR, 2020.

\bibitem{bertsekas2000gradient}
Dimitri~P Bertsekas and John~N Tsitsiklis.
\newblock Gradient convergence in gradient methods with errors.
\newblock {\em SIAM Journal on Optimization}, 10(3):627--642, 2000.

\bibitem{10.1093/acprof:oso/9780199535255.001.0001}
Stéphane Boucheron, Gábor Lugosi, and Pascal Massart.
\newblock {\em Concentration Inequalities: A Nonasymptotic Theory of
  Independence}.
\newblock Oxford University Press, 02 2013.

\bibitem{cutkosky2021high}
Ashok Cutkosky and Harsh Mehta.
\newblock High-probability bounds for non-convex stochastic optimization with
  heavy tails.
\newblock {\em Advances in Neural Information Processing Systems},
  34:4883--4895, 2021.

\bibitem{devolder2011stochastic}
Olivier Devolder et~al.
\newblock Stochastic first order methods in smooth convex optimization.
\newblock Technical report, CORE, 2011.

\bibitem{dodge2020fine}
Jesse Dodge, Gabriel Ilharco, Roy Schwartz, Ali Farhadi, Hannaneh Hajishirzi,
  and Noah Smith.
\newblock Fine-tuning pretrained language models: Weight initializations, data
  orders, and early stopping.
\newblock {\em arXiv preprint arXiv:2002.06305}, 2020.

\bibitem{faw2022power}
Matthew Faw, Isidoros Tziotis, Constantine Caramanis, Aryan Mokhtari, Sanjay
  Shakkottai, and Rachel Ward.
\newblock The power of adaptivity in sgd: Self-tuning step sizes with unbounded
  gradients and affine variance.
\newblock In {\em Conference on Learning Theory}, pages 313--355. PMLR, 2022.

\bibitem{gadat2018stochastic}
S{\'e}bastien Gadat, Fabien Panloup, and Sofiane Saadane.
\newblock Stochastic heavy ball.
\newblock {\em Electronic Journal of Statistics}, 12(1):461--529, 2018.

\bibitem{ge2019stabilized}
Rong Ge, Zhize Li, Weiyao Wang, and Xiang Wang.
\newblock Stabilized {SVRG}: Simple variance reduction for nonconvex
  optimization.
\newblock In {\em Conference on learning theory}, pages 1394--1448. PMLR, 2019.

\bibitem{ghadimi2013stochastic}
Saeed Ghadimi and Guanghui Lan.
\newblock Stochastic first-and zeroth-order methods for nonconvex stochastic
  programming.
\newblock {\em SIAM Journal on Optimization}, 23(4):2341--2368, 2013.

\bibitem{gitman2019understanding}
Igor Gitman, Hunter Lang, Pengchuan Zhang, and Lin Xiao.
\newblock Understanding the role of momentum in stochastic gradient methods.
\newblock {\em Advances in Neural Information Processing Systems}, 32, 2019.

\bibitem{GoodBengCour16}
Ian~J. Goodfellow, Yoshua Bengio, and Aaron Courville.
\newblock {\em Deep Learning}.
\newblock MIT Press, Cambridge, MA, USA, 2016.

\bibitem{gorbunov2020stochastic}
Eduard Gorbunov, Marina Danilova, and Alexander Gasnikov.
\newblock Stochastic optimization with heavy-tailed noise via accelerated
  gradient clipping.
\newblock {\em Advances in Neural Information Processing Systems},
  33:15042--15053, 2020.

\bibitem{harvey2019tight}
Nicholas~JA Harvey, Christopher Liaw, Yaniv Plan, and Sikander Randhawa.
\newblock Tight analyses for non-smooth stochastic gradient descent.
\newblock In {\em Conference on Learning Theory}, pages 1579--1613. PMLR, 2019.

\bibitem{jain2021making}
Prateek Jain, Dheeraj~M Nagaraj, and Praneeth Netrapalli.
\newblock Making the last iterate of {SGD} information theoretically optimal.
\newblock {\em SIAM Journal on Optimization}, 31(2):1108--1130, 2021.

\bibitem{johnson2013accelerating}
Rie Johnson and Tong Zhang.
\newblock Accelerating stochastic gradient descent using predictive variance
  reduction.
\newblock {\em Advances in neural information processing systems}, 26, 2013.

\bibitem{karimi2016linear}
Hamed Karimi, Julie Nutini, and Mark Schmidt.
\newblock Linear convergence of gradient and proximal-gradient methods under
  the polyak-{\l}ojasiewicz condition.
\newblock In {\em Joint European Conference on Machine Learning and Knowledge
  Discovery in Databases}, pages 795--811. Springer, 2016.

\bibitem{kavis2022high}
Ali Kavis, Kfir~Yehuda Levy, and Volkan Cevher.
\newblock High probability bounds for a class of nonconvex algorithms with
  adagrad stepsize.
\newblock In {\em International Conference on Learning Representations}, 2021.

\bibitem{kidambi2018insufficiency}
Rahul Kidambi, Praneeth Netrapalli, Prateek Jain, and Sham Kakade.
\newblock On the insufficiency of existing momentum schemes for stochastic
  optimization.
\newblock In {\em 2018 Information Theory and Applications Workshop (ITA)},
  pages 1--9. IEEE, 2018.

\bibitem{kiefer1952stochastic}
Jack Kiefer and Jacob Wolfowitz.
\newblock Stochastic estimation of the maximum of a regression function.
\newblock {\em The Annals of Mathematical Statistics}, pages 462--466, 1952.

\bibitem{doi:10.1137/0319007}
Harold~J. Kushner and Hai Huang.
\newblock Asymptotic properties of stochastic approximations with constant
  coefficients.
\newblock {\em SIAM Journal on Control and Optimization}, 19(1):87--105, 1981.

\bibitem{lan2012optimal}
Guanghui Lan.
\newblock An optimal method for stochastic composite optimization.
\newblock {\em Mathematical Programming}, 133(1-2):365--397, 2012.

\bibitem{lan2020first}
Guanghui Lan.
\newblock {\em First-order and stochastic optimization methods for machine
  learning}, volume~1.
\newblock Springer, 2020.

\bibitem{lan2012validation}
Guanghui Lan, Arkadi Nemirovski, and Alexander Shapiro.
\newblock Validation analysis of mirror descent stochastic approximation
  method.
\newblock {\em Mathematical programming}, 134(2):425--458, 2012.

\bibitem{li2022high}
Shaojie Li and Yong Liu.
\newblock High probability guarantees for nonconvex stochastic gradient descent
  with heavy tails.
\newblock In {\em International Conference on Machine Learning}, pages
  12931--12963. PMLR, 2022.

\bibitem{li2019convergence}
Xiaoyu Li and Francesco Orabona.
\newblock On the convergence of stochastic gradient descent with adaptive
  stepsizes.
\newblock In {\em The 22nd international conference on artificial intelligence
  and statistics}, pages 983--992. PMLR, 2019.

\bibitem{li2020high}
Xiaoyu Li and Francesco Orabona.
\newblock A high probability analysis of adaptive sgd with momentum.
\newblock In {\em Workshop on Beyond First Order Methods in ML Systems at
  ICML'20}, 2020.

\bibitem{liu2020improved}
Yanli Liu, Yuan Gao, and Wotao Yin.
\newblock An improved analysis of stochastic gradient descent with momentum.
\newblock {\em Advances in Neural Information Processing Systems},
  33:18261--18271, 2020.

\bibitem{liu2023high}
Zijian Liu, Ta~Duy Nguyen, Thien~Hang Nguyen, Alina Ene, and Huy Nguyen.
\newblock High probability convergence of stochastic gradient methods.
\newblock In {\em International Conference on Machine Learning}, pages
  21884--21914. PMLR, 2023.

\bibitem{liu2023revisiting}
Zijian Liu and Zhengyuan Zhou.
\newblock Revisiting the last-iterate convergence of stochastic gradient
  methods.
\newblock {\em arXiv preprint arXiv:2312.08531}, 2023.

\bibitem{madden2020high}
Liam Madden, Emiliano Dall'Anese, and Stephen Becker.
\newblock High-probability convergence bounds for non-convex stochastic
  gradient descent.
\newblock {\em arXiv preprint arXiv:2006.05610}, 2020.

\bibitem{nazin2019algorithms}
Alexander~V Nazin, Arkadi~S Nemirovsky, Alexandre~B Tsybakov, and Anatoli~B
  Juditsky.
\newblock Algorithms of robust stochastic optimization based on mirror descent
  method.
\newblock {\em Automation and Remote Control}, 80:1607--1627, 2019.

\bibitem{nemirovski2009robust}
Arkadi Nemirovski, Anatoli Juditsky, Guanghui Lan, and Alexander Shapiro.
\newblock Robust stochastic approximation approach to stochastic programming.
\newblock {\em SIAM Journal on optimization}, 19(4):1574--1609, 2009.

\bibitem{prechelt2002early}
Lutz Prechelt.
\newblock Early stopping-but when?
\newblock In {\em Neural Networks: Tricks of the trade}, pages 55--69.
  Springer, 2002.

\bibitem{ROBBINS1971233}
H.~Robbins and D.~Siegmund.
\newblock A convergence theorem for non negative almost supermartingales and
  some applications**research supported by nih grant 5-r01-gm-16895-03 and onr
  grant n00014-67-a-0108-0018.
\newblock In Jagdish~S. Rustagi, editor, {\em Optimizing Methods in
  Statistics}, pages 233--257. Academic Press, 1971.

\bibitem{robbins1951stochastic}
Herbert Robbins and Sutton Monro.
\newblock A stochastic approximation method.
\newblock {\em The annals of mathematical statistics}, pages 400--407, 1951.

\bibitem{sebbouh2021almost}
Othmane Sebbouh, Robert~M Gower, and Aaron Defazio.
\newblock Almost sure convergence rates for stochastic gradient descent and
  stochastic heavy ball.
\newblock In {\em Conference on Learning Theory}, pages 3935--3971. PMLR, 2021.

\bibitem{shamir2013stochastic}
Ohad Shamir and Tong Zhang.
\newblock Stochastic gradient descent for non-smooth optimization: Convergence
  results and optimal averaging schemes.
\newblock In {\em International conference on machine learning}, pages 71--79.
  PMLR, 2013.

\bibitem{vershynin2018high}
Roman Vershynin.
\newblock {\em High-dimensional probability: An introduction with applications
  in data science}, volume~47.
\newblock Cambridge university press, 2018.

\bibitem{ville1939etude}
Jean Ville.
\newblock {\em Etude critique de la notion de collectif}.
\newblock Gauthier-Villars Paris, 1939.

\bibitem{wainwright2019high}
Martin~J Wainwright.
\newblock {\em High-Dimensional Statistics: A Non-Asymptotic Viewpoint},
  volume~48.
\newblock Cambridge University Press, 2019.

\bibitem{ward2020adagrad}
Rachel Ward, Xiaoxia Wu, and Leon Bottou.
\newblock Adagrad stepsizes: Sharp convergence over nonconvex landscapes.
\newblock {\em The Journal of Machine Learning Research}, 21(1):9047--9076,
  2020.

\bibitem{yan2018unified}
Yan Yan, Tianbao Yang, Zhe Li, Qihang Lin, and Yi~Yang.
\newblock A unified analysis of stochastic momentum methods for deep learning.
\newblock {\em arXiv preprint arXiv:1808.10396}, 2018.

\bibitem{zhang2020adaptive}
Jingzhao Zhang, Sai~Praneeth Karimireddy, Andreas Veit, Seungyeon Kim, Sashank
  Reddi, Sanjiv Kumar, and Suvrit Sra.
\newblock Why are adaptive methods good for attention models?
\newblock {\em Advances in Neural Information Processing Systems},
  33:15383--15393, 2020.

\end{thebibliography}

\end{document}